\documentclass[11pt]{amsart}
 \usepackage{amsaddr}

\usepackage{latexsym}
\usepackage{fullpage}
\usepackage{amssymb}
\usepackage{amscd}
\usepackage{float}
\usepackage{graphicx}
\usepackage{amsfonts}
\usepackage{pb-diagram}
\usepackage{amsmath,amscd}
\usepackage{subcaption}

\newtheorem{thm}{Theorem}
\newtheorem{cor}{Corollary}
\newtheorem{lemma}{Lemma}
\newtheorem{prop}{Proposition}
\newtheorem{defn}{Definition}
\newtheorem{remark}{Remark}

\newtheorem{nt}{Notation}

\begin{document}

\title[The Kauffman bracket skein module of $S^1\times S^2$ via braids]
  {The Kauffman bracket skein module of $S^1\times S^2$ via braids}

\author{Ioannis Diamantis}
\address{Department of Data Analytics and Digitalisation,
School of Business and Economics, Maastricht University,
P.O.Box 616, 6200 MD, Maastricht,
The Netherlands.
}
\email{i.diamantis@maastrichtuniversity.nl}

\setcounter{section}{-1}

\begin{abstract}
In this paper we present two different ways for computing the Kauffman bracket skein module of $S^1\times S^2$, ${\rm KBSM}\left(S^1\times S^2\right)$, via braids. We first extend the universal Kauffman bracket type invariant $V$ for knots and links in the Solid Torus ST, which is obtained via a unique Markov trace constructed on the generalized Temperley-Lieb algebra of type B, to an invariant for knots and links in $S^1\times S^2$. We do that by imposing on $V$ relations coming from the {\it braid band moves}. These moves reflect isotopy in $S^1\times S^2$ and they are similar to the second Kirby move. We obtain an infinite system of equations, a solution of which, is equivalent to computing ${\rm KBSM}\left(S^1\times S^2\right)$. We show that ${\rm KBSM}\left(S^1\times S^2\right)$ is not torsion free and that its free part is generated by the unknot (or the empty knot). We then present a diagrammatic method for computing ${\rm KBSM}\left(S^1\times S^2\right)$ via braids. Using this diagrammatic method we also obtain a closed formula for the torsion part of ${\rm KBSM}\left(S^1\times S^2\right)$.

\smallbreak
\bigbreak


\end{abstract}


\maketitle

\tableofcontents

\section{Introduction and overview}\label{intro}

Skein modules of $3$-manifolds are generalizations of link invariants in $S^3$ to link invariants in arbitrary $3$-manifolds. They were introduced by Turaev \cite{Tu} and Przytycki \cite{P} and they have become very important algebraic tools in the study of 3-manifolds, since their properties reflect geometric and topological information about them. Skein modules of 3-manifolds are modules composed of linear combinations of links in the 3-manifolds, modulo some properly chosen (local) skein relations. In this paper we compute the {\it Kauffman bracket skein module} of $S^1\times S^2$, KBSM($S^1\times S^2$), namely, the skein module based on the Kauffman bracket skein relation 
\[
L_+-AL_{0}-A^{-1}L_{\infty}
\] 
\noindent where $L_{\infty}$ and $L_{0}$ are represented schematically by the illustrations in Figure~\ref{skein}. 

\smallbreak

The computation of KBSM($S^1\times S^2$) is equivalent to constructing all possible analogues of the Jones polynomial for knots and links in
$S^1\times S^2$, since the linear dimension of KBSM($S^1\times S^2$) means the number of independent Kauffman bracket-type invariants defined on knots and links in $S^1\times S^2$. In other words, KBSM($S^1\times S^2$) yields all possible isotopy invariants of knots which satisfy the Kauffman bracket skein relation.

\smallbreak

The precise definition of KBSM is as follows:

\begin{defn}\rm
Let $M$ be an oriented $3$-manifold and $\mathcal{L}_{{\rm fr}}$ be the set of isotopy classes of unoriented framed links in $M$. Let $R=\mathbb{Z}[A^{\pm1}]$ be the Laurent polynomials in $A$ and let $R\mathcal{L}_{{\rm fr}}$ be the free $R$-module generated by $\mathcal{L}_{{\rm fr}}$. Let $\mathcal{S}$ be the ideal generated by the skein expressions $L_+-AL_{0}-A^{-1}L_{\infty}$ and $L \bigsqcup {\rm O} - (-A^2-A^{-2})L$. Note that blackboard framing is assumed and that $L \bigsqcup {\rm O}$ stands for the union of a link $L$ and the trivially framed unknot in a ball disjoint from $L$. 

\begin{figure}[H]
\begin{center}
\includegraphics[width=2.1in]{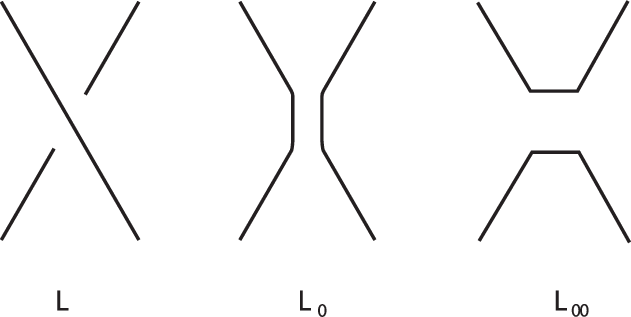}
\end{center}
\caption{The links $L$, $L_{0}$ and $L_{\infty}$ locally.}
\label{skein}
\end{figure}

\noindent Then the {\it Kauffman bracket skein module} of $M$, KBSM$(M)$, is defined to be:

\begin{equation*}
{\rm KBSM} \left(M\right)={\raise0.7ex\hbox{$
R\mathcal{L}_{{\rm fr}} $}\!\mathord{\left/ {\vphantom {R\mathcal{L_{{\rm fr}}} {\mathcal{S} }}} \right. \kern-\nulldelimiterspace}\!\lower0.7ex\hbox{$ S  $}}.
\end{equation*}

\end{defn}

The importance of the Kauffman bracket skein module of a 3-manifold $M$ lies in the fact that it detects the presence of non-separating 2-spheres and tori embedded in $M$. On the level of Kauffman bracket skein modules, a non-separating 2-sphere in $M$ is detected by the existence of torsion in KBSM($M$).

\bigbreak

We consider $S^1\times S^2$ as being obtained from $S^3$ by integral surgery along the unknot with coefficient zero, that is, $S^1\times S^2$ is a special type of lens spaces $L(p, q)$ for $p=0$ and $q=1$ ($L(0,1)\, \cong\, S^1\times S^2$). In \cite{D0}, the Kauffman bracket skein module of the lens spaces $L(p, q),\, p>0$ is computed via the `braid technique', and a basis of KBSM($L(p,q)$) is obtained that is different from the basis presented in \cite{HP}. In this paper we apply the `braid technique' to exhaust the computation of the Kauffman bracket skein module of the family of lens spaces. More precisely, the integral surgery model of $S^1\times S^2$ allows us to `push' every knot and link in $S^1\times S^2$ to the solid torus, and thus, a basis of KBSM(ST) spans KBSM($S^1\times S^2$). Following the pioneering work of V.F.R. Jones (\cite{Jo, Jo1}), in \cite{D0} we consider the generalized Temperley-Lieb algebra of type B, $TL_{1, n}$, together with a Markov trace constructed on $TL_{1, n}$ and the universal invariant $V$ of the Kauffman bracket type for knots and links in ST defined in \cite{D0}. This invariant gives distinct values to distinct elements of any basis of KBSM(ST), and thus, it recovers KBSM(ST). Hence, in order to compute KBSM($S^1\times S^2$), one needs to extend the invariant $V$ for knots and links in ST to an invariant of knots and links in $S^1\times S^2$, by imposing relations coming from the extra isotopy moves of $S^1\times S^2$, that we call {\it (braid) band moves}, abbreviated $bm$ (or $bbm$). In other words, the following infinite system of equations is obtained:

\begin{equation}\label{eqbbm}
V_{a}\ =\ V_{bm(a)},\quad {\rm for\ every}\ a\ {\rm in\ a\ basis\ of}\ {\rm KBSM(ST)}.
\end{equation}

\noindent By construction, a solution to the infinite system of Equations~(\ref{eqbbm}) corresponds to the computation of the Kauffman bracket skein module of $S^1\times S^2$. These equations are simplified with the use of a new basis of KBSM(ST), $B_{{\rm ST}}$, presented in \cite{D0, D1, GM}, that arises naturally on the level of braids. Using the basis $B_{{\rm ST}}$ of KBSM(ST), we show that the free part of the Kauffman bracket skein module of $S^1\times S^2$ is generated by the unknot (or the empty knot) and we also show how torsion is captured in the solution of the infinite system. It is worth mentioning that $S^1\times S^2$ is the first $3$-manifold, where torsion on its Kauffman bracket skein module is detected via the `braid technique'. Finally, we compute KBSM($S^1\times S^2$) via a different diagrammatic method based on braids following \cite{D4, D2}. In this way we obtain a closed formula for the torsion part of the module and our result agrees with that of Hoste and Przytycki \cite{HP1}. The diagrammatic method via braids has been successfully applied in \cite{D4} for the case of the Kauffman bracket skein module of the complement of $(2, 2p+1)$-torus knots and in \cite{D0} for the case of the lens spaces $L(p,q)$, $p\neq 0$. The importance of the braid approach lies in the fact that it can shed light to the problem of computing (various) skein modules of arbitrary c.c.o. $3$-manifolds (see also \cite{DL3} for the case of the HOMFLYPT skein module of the lens spaces $L(p,1)$). 

\bigbreak

The paper is organized as follows: In \S\ref{basics1} we recall the topological setting and the essential techniques and results from \cite{LR1, DL1}. We first present different models for $S^1 \times S^2$, which demonstrate the difference between the knot theory of this $3$-manifold to the knot theory of $S^3$. We then explain how knots and links in $S^1 \times S^2$ can be viewed as mixed links in $S^3$. Mixed links are close related to Kirby diagrams of links in $S^3$, that is, link diagrams in $S^3$ that involve the surgery description of $S^1 \times S^2$. We present isotopy moves for mixed links in $S^3$ that correspond to isotopy moves for knots and links in $S^1 \times S^2$, and using the analogue of the Alexander theorem for knots and links in ST, we pass on the level of mixed braids. We then translate isotopy in $S^1 \times S^2$ to braid equivalence (the analogue of the Markov theorem for braid equivalence in $S^1 \times S^2$). In \S\ref{basics2} we present results from \cite{La1, La2} and \cite{D0, D1}. More precisely, we start by recalling results on the generalized Hecke algebra of type B, $H_{1, n}$, and we present the universal invariant of the HOMFLYPT type for knots in ST via a unique Markov trace, $tr$, defined on $H_{1, n}$. This invariant captures the HOMFLYPT skein module of ST. We then pass to the {\it generalized Temperley-Lieb algebra of type B}, defined as a quotient algebra of $H_{1, n}$ over an appropriate chosen ideal (see relations (\ref{ideal})), and we obtain the universal invariant $V$ for knots and links in ST of the Kauffman bracket type, that captures KBSM(ST). We also present the basis $B_{\rm ST}$ of KBSM(ST) that describes naturally the isotopy moves for knots and links in $S^1 \times S^2$, and that simplifies the infinite system of equations (\ref{eqbbm}). In \S~\ref{infs}, we solve the infinite system of equations (\ref{eqbbm}) and we prove that the free part of KBSM($S^1 \times S^2$) is generated by the unknot (or the empty knot). We also show how torsion is detected in the solution of this system. Finally, in \S~\ref{kbsmunbr} we demonstrate the diagrammatic approach based on braids for computing KBSM($S^1 \times S^2$).

%\newpage

\section{Topological Set Up}\label{basics1}

\subsection{Models of $S^1 \times S^2$}\label{mod}

$S^1\times S^2$ can be realized as a thickened torus in $4$-dimensions, such that the cross-sections are two copies of the $2$-sphere $S^2$. To better perceive this, we consider the analogue construction of $S^1 \times S^1$ in dimension three: we consider copies of $S^1$ and we `glue' them together to obtain a cylinder. Finally, we glue the first and the last $S^1$ to obtain a torus (see Figure~\ref{s1s1}(a) and (b)). This construction is equivalent to `gluing' the inner and outer $S^1$'s from an infinitum concentric $S^1$'s (see Figure~\ref{s1s1}(c)). Analogously, $S^1 \times S^2$ may be visualized as a three ball $B^3$ with a concentric and of smaller radius $S^2$ being hollowed out, and where the boundary of $B^3$ is glued to the inner $S^2$ as illustrated in Figure~\ref{s1s1}(d). 

\begin{figure}[H]
\begin{center}
\includegraphics[width=6in]{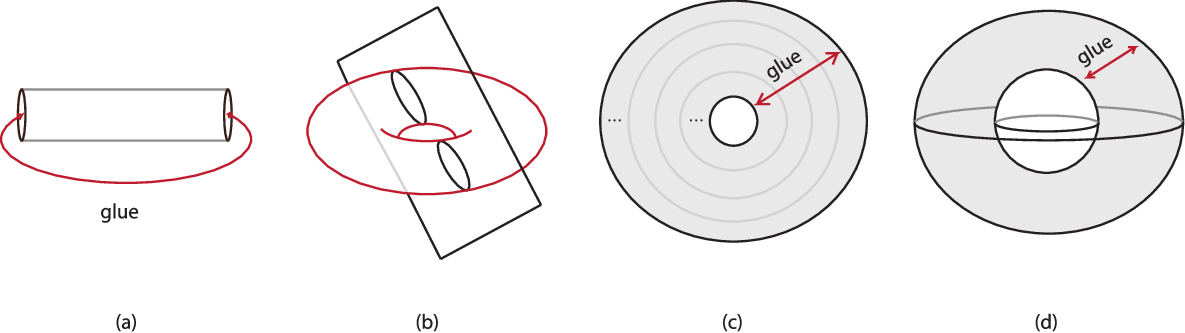}
\end{center}
\caption{Models of $S^1 \times S^1$ (a, b, c) and $S^1 \times S^2$ (d).}
\label{s1s1}
\end{figure}

The model of $S^1 \times S^2$ as illustrated in Figure~\ref{s1s1}(d) can be used to highlight the difference between the knot theory of $S^1 \times S^2$ and the knot theory of $S^3$. Consider for example the knot in Figure~\ref{unks1s2} and notice that using isotopy, we may unknot this knot with the use of the non-separating $2$-sphere. This method is known as the {\it Dirac trick}.

\begin{figure}[H]
\begin{center}
\includegraphics[width=6.2in]{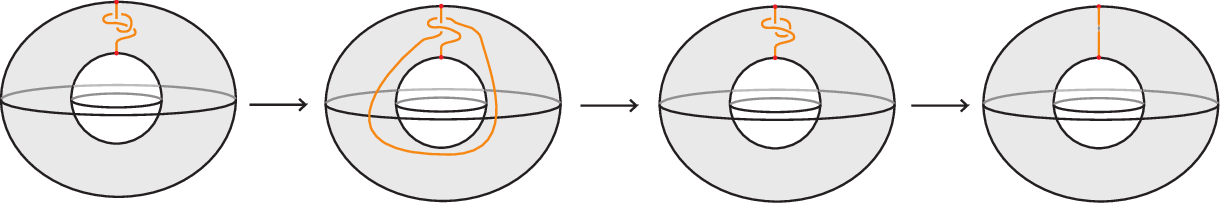}
\end{center}
\caption{The Dirac trick for unknotting a knot in $S^1 \times S^2$.}
\label{unks1s2}
\end{figure}

A different way to visualize $S^1\times S^2$ is by `gluing' two solid tori via some homeomorphism $h$ on their boundaries, such that $h$ takes a meridian curve $m_1$ on the first solid torus ${\rm ST}_1$, to the corresponding meridian curve $m_2$ of the second solid torus ${\rm ST}_2$. We may choose $h$ to be the identity $i: \partial {\rm ST}_1\, \rightarrow\, \partial {\rm ST}_2$. Since now $ST\, =\, S^1\times D^2$, where $D^2$ denotes a disc, and since the gluing of two discs results in a 2-sphere $S^2$, the resulting $3$-manifold can be considered as a family of $2$-spheres parametrized by a circle (for an illustration see Figure~\ref{stst}). 

\begin{figure}[H]
\begin{center}
\includegraphics[width=4.2in]{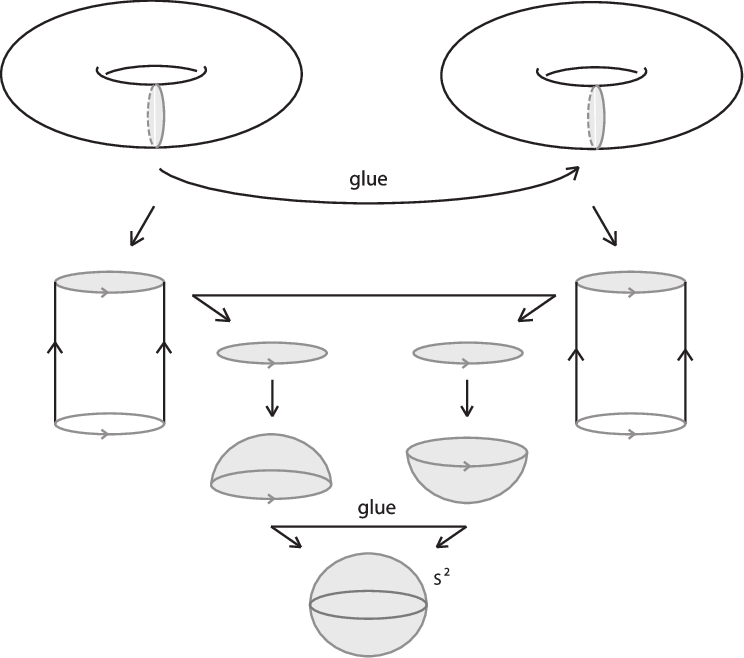}
\end{center}
\caption{Gluing two solid tori to obtain $S^1 \times S^2$.}
\label{stst}
\end{figure}

The above construction of $S^1 \times S^2$ coincides with the Dehn surgery method for obtaining $S^1 \times S^2$. This is a method for obtaining any $3$-manifold by removing finitely many disjoint solid tori from $S^3$ and then sewing them back in a different way. Recall now that by Alexander's trick, the $3$-sphere $S^3$ is the union of two solid tori via a homeomorphism $h$ on their boundaries that maps a meridian curve on the first solid torus to a longitude curve on the second. Hence, $S^1 \times S^2$ may be obtained by first removing a solid torus ${\rm ST}_1$ from $S^3$, and since $S^3\, =\, {\rm ST}_1\, \underset{h}{\bigcup}\, {\rm ST}_2$, what remains is another solid torus ${\rm ST}_2$. We finally glue ${\rm ST}_1$ to ${\rm ST}_2$ via $h: m_1 \mapsto m_2$, where $m_1$ is a meridian on ${\rm ST}_1$ and $m_2$ is a meridian on ${\rm ST}_2$. Let us present the above results in a more rigorous format. Consider the tubular neighborhood of the unknot in $S^3$. This is obviously a solid torus, say ${\rm ST}_1$. Then, this solid torus is to be ``drilled away'' from $S^3$ and then sewed back via $h$. $h$ can be fully determined by a coefficient that instructs how ${\rm ST}_1$ is to be sewed back. In the case of $S^1 \times S^2$, the meridian of ${\rm ST}_1$ is to be glued to a meridian of ${\rm ST}_2$, that is a $(0, 1)$ curve on the boundary of ${\rm ST}_2$. Thus, we say that $S^1 \times S^2$ is obtained from $S^3$ by surgery along the unknot with coefficient zero, i.e. $S^1\times S^2 \cong {\rm O}^0$. It is worth mentioning that $S^1 \times S^2$ may be considered as a special type of lens spaces $L(p,q)$, where $p=0$ and $q=1$.

\subsection{Knots and braids in $S^1\times S^2$}

In this subsection we recall results from \cite{LR1, LR2, DL1}. As mentioned before, we shall consider ST to be the complement of another solid torus in $S^3$ and an oriented link $L$ in ST can be represented by an oriented \textit{mixed link} in $S^{3}$. A mixed link is a link in $S^{3}$ consisting of the unknotted fixed part $\widehat{I}$ representing the complementary solid torus in $S^3$, and the moving part $L$ that links with $\widehat{I}$. A \textit{mixed link diagram} is a diagram $\widehat{I}\cup \widetilde{L}$ of $\widehat{I}\cup L$ on the plane of $\widehat{I}$, where this plane is equipped with the top-to-bottom direction of $I$ (for an illustration see Figure~\ref{mli}).

\begin{figure}[H]
\begin{center}
\includegraphics[width=3.2in]{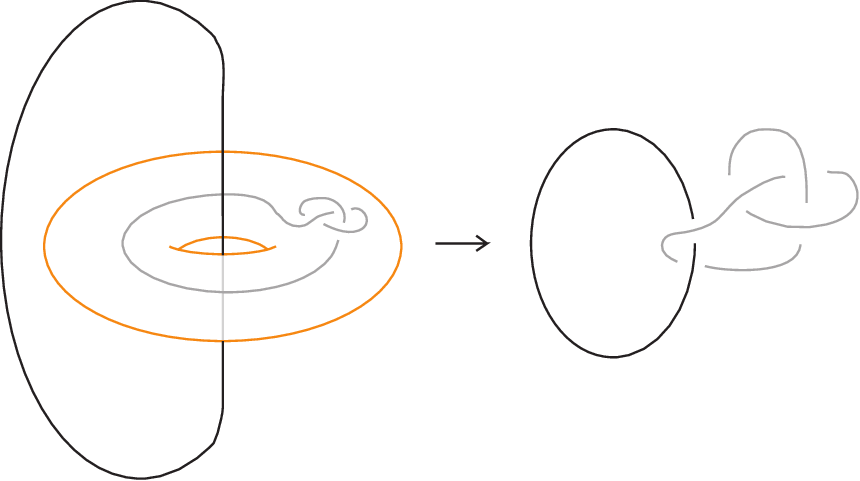}
\end{center}
\caption{ A mixed link in $S^3$ representing a link in ST. }
\label{mli}
\end{figure}

Isotopy of oriented links in ST consists in isotopy in $S^3$ for the moving part of the mixed link, together with the {\it mixed Reidemeister moves}, that involve both the fixed and the moving part of the mixed link (see Figure~\ref{r-thm}).

\begin{figure}[H]
\begin{center}
\includegraphics[width=3.4in]{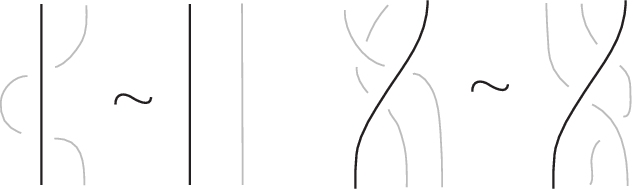}
\end{center}
\caption{ The mixed Reidemeister moves.}
\label{r-thm}
\end{figure}

\smallbreak

As mentioned above, $S^1 \times S^2$ can be obtained from $S^3$ by surgery on the unknot with surgery coefficient zero. In other words, $S^1 \times S^2$ can be obtained by ``killing'' the meridian $m$, that is, by mirroring $D \times S^1$ along its boundary. Knots and links in $S^1 \times S^2$ may be then visualized in $S^3$ as mixed links, together with the surgery coefficient on their fixed component (``Kirby'' like diagrams). It follows that isotopy in $S^1 \times S^2$ can be then viewed as isotopy in ST together with the {\it band moves}, which reflect the surgery description of the manifold and that are similar to the second Kirby move. It is worth mentioning that there are two types of band moves, according to the orientation of the component of the knot and that of the surgery curve (see Figure~\ref{bmov} for the two types of band moves for the case of $S^1 \times S^2$). In the $\beta$-type moves the orientation of the arc is initially opposite to the orientation of the surgery curve, but after the performance of the move their orientations agree. In the $\alpha$-type the orientations initially agree, but disagree after the performance of the move. 

\begin{figure}[H]
\begin{center}
\includegraphics[width=5.5in]{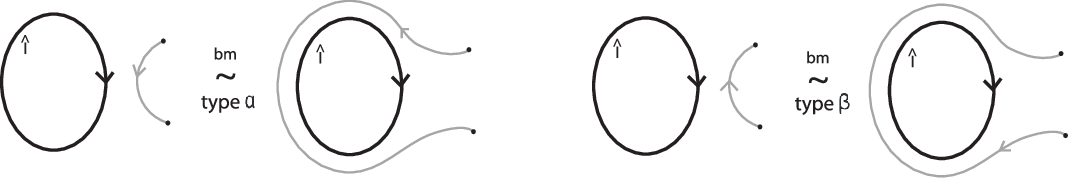}
\end{center}
\caption{The two types of band moves.}
\label{bmov}
\end{figure}

In \cite{DL1} it is shown that in order to describe isotopy for knots and links in a c.c.o. $3$-manifold, it suffices to consider only one type of band moves (cf. \cite[Theorem 6]{DL1}) and thus, isotopy between oriented links in $S^1 \times S^2$ is reflected in $S^3$ by means of the following theorem:

\begin{thm}[{\bf The analogue of Reidemeister's theorem}]
Two oriented links in $S^1 \times S^2$ are isotopic if and only if two corresponding mixed link diagrams of theirs differ by isotopy in {\rm ST} together with a finite sequence of one type of band moves.
\end{thm}

As shown in \cite{La1}, every mixed link is isotopic to the closure of a {\it mixed braid}. A mixed braid in $S^3$ is a standard braid where, without loss of generality, its first strand represents $\widehat{I}$, the fixed part, and the other strands, $\beta$, represent the moving part. We shall call the subbraid $\beta$ the \textit{moving part} of $I\cup \beta$ (see bottom left illustration of Figure~\ref{bbmov}).

\smallbreak

We now recall the analogue of the Alexander theorem for knots and links in ST (cf. \cite[Theorem 1]{La2}):

\begin{thm}[{\bf The analogue of Alexander's theorem}]
A mixed link diagram $\widehat{I}\cup \widetilde{L}$ of $\widehat{I}\cup L$ may be turned into a \textit{mixed braid} $I\cup \beta$ with isotopic closure.
\end{thm}

Hence, a link in $S^1 \times S^2$, that is visualized as a mixed link in $S^3$, may be turned to a mixed braid in $S^3$. We now translate isotopy in $S^1 \times S^2$ on the level of mixed braids. We define a {\it braid band move} (abbreviated to bbm) to be a move between mixed braids, which is a $\beta$-type band move between their closures. It starts with a little band oriented downward, which, before sliding along a surgery strand, gets one twist {\it positive\/} or {\it negative\/} (see middle part of Figure~\ref{bbmov}). Note that in \cite{LR2} it was shown that the choice of the position of connecting the two components after the performance of a bbm is arbitrary.

\begin{figure}[H]
\begin{center}
\includegraphics[width=4.5in]{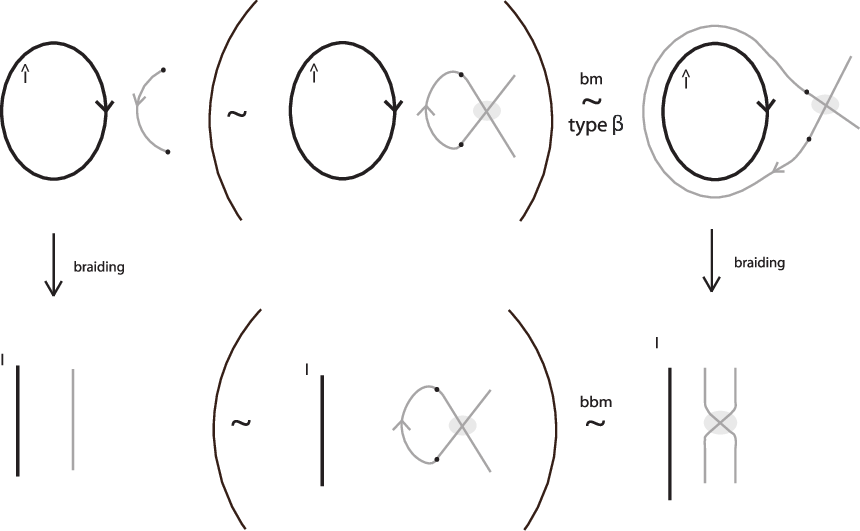}
\end{center}
\caption{ The two types of braid band moves.}
\label{bbmov}
\end{figure}

\smallbreak

The sets of braids related to ST form groups, which are in fact the Artin braid groups of type B, denoted $B_{1,n}$, with presentation:

\[ B_{1,n} = \left< \begin{array}{ll}  \begin{array}{l} t, \sigma_{1}, \ldots ,\sigma_{n-1}  \\ \end{array} & \left| \begin{array}{l}
\sigma_{1}t\sigma_{1}t=t\sigma_{1}t\sigma_{1} \ \   \\
 t\sigma_{i}=\sigma_{i}t, \quad{i>1}  \\
{\sigma_i}\sigma_{i+1}{\sigma_i}=\sigma_{i+1}{\sigma_i}\sigma_{i+1}, \quad{ 1 \leq i \leq n-2}   \\
 {\sigma_i}{\sigma_j}={\sigma_j}{\sigma_i}, \quad{|i-j|>1}  \\
\end{array} \right.  \end{array} \right>, \]

\noindent where the generators $\sigma _{i}$ and $t$ are illustrated in Figure~\ref{genh}(i). Notice that the mixed braid group $B_{1, n}$ can be considered as a subgroup of the classical braid group on $(n+1)$ strands, $B_{n+1}$, involving elements whose first strand is fixed.

\smallbreak

Let now $\mathcal{L}$ denote the set of oriented knots and links in $S^1 \times S^2$. Isotopy in $S^1 \times S^2$ is then translated on the level of mixed braids by means of the following theorem (see also \cite[Theorem~5]{LR2} and \cite[Theorem~9]{DL1}):

\begin{thm}\label{markov}
 Let $L_{1} ,L_{2}$ be two oriented links in $S^1 \times S^2$ and let $I\cup \beta_{1} ,{\rm \; }I\cup \beta_{2}$ be two corresponding mixed braids in $S^{3}$. Then $L_{1}$ is isotopic to $L_{2}$ in $S^1 \times S^2$ if and only if $I\cup \beta_{1}$ is equivalent to $I\cup \beta_{2}$ in $\mathcal{B}$ by the following moves:
\[ \begin{array}{clll}
(i)  & Conjugation:         & \alpha \sim \beta^{-1} \alpha \beta, & {\rm if}\ \alpha ,\beta \in B_{1,n}. \\
(ii) & Stabilization\ moves: &  \alpha \sim \alpha \sigma_{n}^{\pm 1} \in B_{1,n+1}, & {\rm if}\ \alpha \in B_{1,n}. \\
(iii) & Loop\ conjugation: & \alpha \sim t^{\pm 1} \alpha t^{\mp 1}, & {\rm if}\ \alpha \in B_{1,n}. \\
(iv) & Braid\ band\ moves: & \alpha \sim  \alpha_+ \sigma_1^{\pm 1}, & a_+\in B_{1, n+1},
\end{array}
\]

\noindent where $\alpha_+$ is the word $\alpha$ with all indices shifted by +1. Note that moves (i), (ii) and (iii) correspond to link isotopy in {\rm ST}.
\end{thm}

\begin{nt}\rm
We denote a braid band move by bbm and, specifically, the result of a positive (cor. negative) braid band move performed on a mixed braid $\beta$ by $bbm_{+}(\beta)$ (cor. $bbm_{-}(\beta)$).
\end{nt}

\section{Algebraic Tools}\label{basics2}

In this section we recall results from \cite{D0} that are crucial for the computation of KBSM($S^1\times S^2$). More precisely, we first recall results on the generalized Temperley-Lieb algebra of type B, $TL_{1, n}$, that is related to the knot theory of ST, and through $TL_{1, n}$ we obtain the universal invariant for knots and links in ST of the Kauffman bracket type. We also present different bases of KBSM(ST) via braids.

\subsection{Knot algebras related to ST \& $S^1\times S^2$}

It is well understood that the generalized Hecke algebra of type B is related to the knot theory of the solid torus \cite{La2}. This algebra is defined as a quotient of the mixed braid group algebra ${\mathbb Z}\left[q^{\pm 1} \right]B_{1,n}$ over the ideal generated by the quadratic relations ${\sigma_{i}^2=(q-1)\sigma_{i}+q}$. Namely:

\begin{equation*}
H_{1,n}(q)= \frac{{\mathbb Z}\left[q^{\pm 1} \right]B_{1,n}}{ \langle \sigma_i^2 -\left(q-1\right)\sigma_i-q \rangle}.
\end{equation*}

\noindent As shown in \cite{La2}, $H_{1,n}$ admits a unique Markov trace, $tr$. This Markov trace gives rise to the most generic analogue of the HOMFLYPT polynomial, $X$, for links in the solid torus ST. The construction of $tr$ was possible with the use of appropriate inductive bases on  $\mathcal{H}:=\bigcup_{n=1}^{\infty}\, H_{1, n}$. Namely:

\begin{thm}[Proposition~1 \& Theorem~1 \cite{La2}]
The following sets form linear bases for ${\rm H}_{1,n}(q)$:
\begin{equation}\label{bHeck}
\begin{array}{llll}
 (i) & \Sigma_{n} & = & \{t_{i_{1} } ^{k_{1} } \ldots t_{i_{r}}^{k_{r} } \cdot \sigma \} ,\ {\rm where}\ 0\le i_{1} <\ldots <i_{r} \le n-1,\\
 (ii) & \Sigma^{\prime} _{n} & = & \{ {t^{\prime}_{i_1}}^{k_{1}} \ldots {t^{\prime}_{i_r}}^{k_{r}} \cdot \sigma \} ,\ {\rm where}\ 0\le i_{1} < \ldots <i_{r} \le n-1, \\
\end{array}
\end{equation}
\noindent where $k_{1}, \ldots ,k_{r} \in {\mathbb Z}$, $t_0^{\prime}\ =\ t_0\ :=\ t, \quad t_i^{\prime}\ =\ g_i\ldots g_1tg_1^{-1}\ldots g_i^{-1} \quad {\rm and}\quad t_i\ =\ g_i\ldots g_1tg_1\ldots g_i$ are the `looping elements' in ${\rm H}_{1, n}(q)$ (see Figure~\ref{genh}(ii)) and $\sigma$ a basic element in the Iwahori-Hecke algebra of type A, ${\rm H}_{n}(q)$.
\end{thm}

Note that a basic element in ${\rm H}_{n}(q)$ is of the form of elements in the following set \cite{Jo}:

$$ S_n =\left\{(g_{i_{1} }g_{i_{1}-1}\ldots g_{i_{1}-k_{1}})(g_{i_{2} }g_{i_{2}-1 }\ldots g_{i_{2}-k_{2}})\ldots (g_{i_{p} }g_{i_{p}-1 }\ldots g_{i_{p}-k_{p}})\right\}, $$

\noindent for $1\le i_{1}<\ldots <i_{p} \le n-1{\rm \; }$.

\begin{figure}[H]
\begin{center}
\includegraphics[width=5.5in]{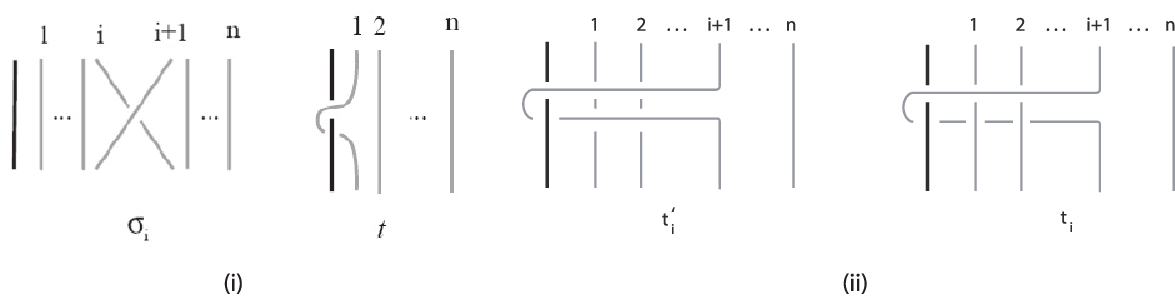}
\end{center}
\caption{The generators of $B_{1, n}$ and the `looping' elements $t^{\prime}_{i}$ and $t_{i}$.}
\label{genh}
\end{figure}

\begin{thm}{\cite[Theorem~6 \& Definition~1]{La2}} \label{tr}
Given $z, s_{k}$ with $k\in {\mathbb Z}$ specified elements in $R={\mathbb Z}\left[q^{\pm 1} \right]$, there exists a unique linear Markov trace function on $\mathcal{H}$:

\begin{equation*}
{\rm tr}:\mathcal{H}  \to R\left(z,s_{k} \right),\ k\in {\mathbb Z}
\end{equation*}

\noindent determined by the rules:

\[
\begin{array}{lllll}
(1) & {\rm tr}(ab) & = & {\rm tr}(ba) & \quad {\rm for}\ a,b \in {\rm H}_{1,n}(q) \\
(2) & {\rm tr}(1) & = & 1 & \quad {\rm for\ all}\ {\rm H}_{1,n}(q) \\
(3) & {\rm tr}(ag_{n}) & = & z{\rm tr}(a) & \quad {\rm for}\ a \in {\rm H}_{1,n}(q) \\
(4) & {\rm tr}(a{t^{\prime}_{n}}^{k}) & = & s_{k}{\rm tr}(a) & \quad {\rm for}\ a \in {\rm H}_{1,n}(q),\ k \in {\mathbb Z} \\
\end{array}
\]

\bigbreak

\noindent Then, the function $X:\mathcal{L}$ $\rightarrow R(z,s_{k})$

\begin{equation*}
X_{\widehat{\alpha}} = \Delta^{n-1}\cdot \left(\sqrt{\lambda } \right)^{e}
{\rm tr}\left(\pi \left(\alpha \right)\right),
\end{equation*}

\noindent is an invariant of oriented links in {\rm ST}, where $\Delta:=-\frac{1-\lambda q}{\sqrt{\lambda } \left(1-q\right)}$, $\lambda := \frac{z+1-q}{qz}$, $\alpha \in B_{1,n}$ is a word in the $\sigma _{i}$'s and $t^{\prime}_{i} $'s, $\widehat{\alpha}$ is the closure of $\alpha$, $e$ is the exponent sum of the $\sigma _{i}$'s in $\alpha $, $\pi$ the canonical map of $B_{1,n}$ to ${\rm H}_{1,n}(q)$, such that $t\mapsto t$ and $\sigma _{i} \mapsto g_{i}$.
\end{thm}

\begin{remark}\rm
As shown in \cite{La2, DL2} the invariant $X$ recovers the HOMFLYPT skein module of ST (see also \cite{DL3, DGLM} for a survey on the HOMFLYPT skein module of the lens spaces $L(p, 1)$ via braids).
\end{remark}

The algebraic counterpart construction of the Kauffman bracket is the Temperley-Lieb algebra together with a unique Markov trace constructed by V.F.R. Jones (see \cite{Jo} and references therein). It is therefore natural to consider the {\it generalized Temperley-Lieb algebra of type B}, defined as a quotient of $H_{1,n}(q)$, in order to obtain a universal invariant of the Kauffman bracket type for knots and links in ST. We have the following definition (\cite[Definition~2]{D0}):

\begin{defn}\rm
The {\it generalized Temperley-Lieb algebra of type B}, $TL_{1, n}$, is defined as the quotient of the generalized Hecke algebra of type B, $H_{1,n}(q)$, over the ideal generated by expressions of the form: 
\begin{equation}\label{ideal}
\begin{array}{lcl}
\sigma_{i, i+1} & := & 1 + u\ (\sigma_i+ \sigma_{i+1}) + u^2\ (\sigma_i\sigma_{i+1}+\sigma_{i+1}\sigma_i)+u^3\ \sigma_i\sigma_{i+1}\sigma_i.\\
\end{array}
\end{equation} 
\end{defn}

\begin{remark}\rm
For the rest of the paper we will adapt a different presentation for $H_{1, n}$, that involves the parameter $u$ and the quadratic relations: 
\begin{equation}\label{quad}
{\sigma_{i}^2=(u-u^{-1})\, \sigma_{i}+1}.
\end{equation}
Note also that we will be using the $\sigma_i$'s for the trace instead of the $g_i$'s. Finally, note that we will be using the symbol $\widehat{\cong}$ when both stabilization moves and conjugation are performed (instead of using tr).
\end{remark}

In \cite{D0}, it is shown that for $z=-\frac{1}{u(1+u^2)}$ the Markov trace constructed in $H_{1, n}$ factors through to ${TL}_{1, n}$. This leads to the universal invariant of the Kauffman bracket type for knots and links in ST. More precisely, we have the following theorem:

\begin{thm}\label{inva}
The invariant

\begin{equation*}
V_{\widehat{\alpha}}(u)\ :=\ \left(-\frac{1+u^2}{u} \right)^{n-1} \left(u\right)^{2e} {\rm tr}\left(\overline{\pi} \left(\alpha \right)\right),
\end{equation*}

\noindent where $\alpha \in B_{1,n}$ is a word in the $\sigma _{i}$'s and $t^{\prime}_{i} $'s, $\widehat{\alpha}$ is the closure of $\alpha$, $e$ is the exponent sum of the $\sigma _{i}$'s in $\alpha $, $\overline{\pi}$ the canonical map of $B_{1,n}$ to ${\rm TL}_{1, n}$, such that $t\mapsto t$ and $\sigma _{i} \mapsto g_{i}$, is the universal invariant for knots and links in ST of the Kauffman bracket type.
\end{thm}

The invariant $V$ recovers the Kauffman bracket skein module of ST and as shown in \cite{D0}, it can be extended to an invariant for knots and links in $L(p,q)$ by considering the effect of the braid band moves on elements in some basis of KBSM(ST) (see also \cite{P}). Equivalently, and since $S^1 \times S^2\cong L(0,1)$, in order to compute KBSM($S^1 \times S^2$), it suffices to consider elements in some basis of KBSM(ST), impose on $V$ relations coming from the band moves and solve the resulting infinite system of equations.

\subsection{The Kauffman bracket skein module of ST via braids}

From the discussion in \S~\ref{basics1}, it is clear that the knot theory of ST is of fundamental importance for studying knot theory in other c.c.o. $3$-manifolds. In \cite{HK}, the Kauffman bracket skein module of the solid torus is computed using diagrammatic methods by means of the following theorem:

\begin{thm}[\cite{HK}]
The Kauffman bracket skein module of ST, KBSM(ST), is freely generated by an infinite set of generators $\left\{x^n\right\}_{n=0}^{\infty}$, where $x^n$ denotes a parallel copy of $n$ longitudes of ST and $x^0$ is the affine unknot (see Figure~\ref{tur}).
\end{thm}

\begin{figure}[H]
\begin{center}
\includegraphics[width=5in]{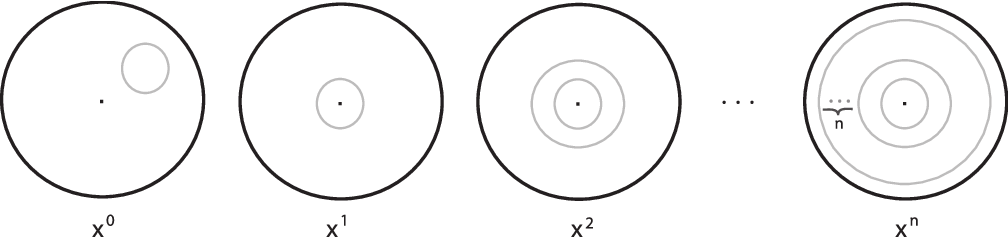}
\end{center}
\caption{The standard basis of KBSM(ST).}
\label{tur}
\end{figure}

In the braid setting, elements in the standard basis of KBSM(ST) correspond bijectively to the elements of the following set (for an illustration see Figure~\ref{tur1}):

\begin{equation}\label{Lpr}
\mathcal{B}^{\prime}_{{\rm ST}}=\{ t{t^{\prime}_1} \ldots {t^{\prime}_n}, \ n\in \mathbb{N} \}.
\end{equation}

\noindent In other words, the set $\mathcal{B}^{\prime}_{{\rm ST}}$ forms a basis for KBSM(ST) in terms of braids.

\begin{figure}[H]
\begin{center}
\includegraphics[width=4.7in]{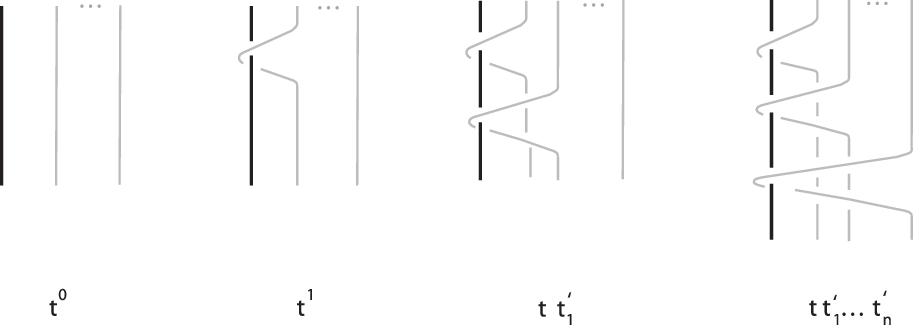}
\end{center}
\caption{The standard basis of KBSM(ST) in terms of mixed braids.}
\label{tur1}
\end{figure}

\begin{remark}\label{rm11}\rm
\begin{itemize}
\item[i.] The basis $\mathcal{B}^{\prime}_{{\rm ST}}$ is a subset of $\mathcal{H}$ and, in particular, $\mathcal{B}^{\prime}_{{\rm ST}}$ is a subset of $\Sigma^{\prime}=\bigcup_n\Sigma^{\prime}_n$. Moreover, and in contrast to elements in $\Sigma^{\prime}$, the elements in $\mathcal{B}^{\prime}_{{\rm ST}}$ have no `gaps' in the indices, the exponents are all equal to one and there are no `braiding tails'. 
\smallbreak
\item[ii.] The invariant $V$ defined in Theorem~\ref{inva} recovers KBSM(ST), since it gives distinct values to distinct elements of $\mathcal{B}^{\prime}_{{\rm ST}}$. Indeed, we have that ${\rm tr}(t{t^{\prime}_1} \ldots {t^{\prime}_n})=s_{1}^{n+1}$.
\end{itemize}
\end{remark}

\bigbreak

It follows from Remark~\ref{rm11}(ii), that in order to compute the Kauffman bracket skein module of $S^1\times S^2$, it suffices to extend the (most generic) invariant $V$ for knots and links in ST, following the ideas in \cite{DL2, DL3, DL4}. That is, it suffices to solve the infinite system of Equations~(\ref{eqbbm}), $V_{\widehat{\alpha}}\ =\  V_{\widehat{bbm_{\pm}(\alpha)}}$, for all $\alpha$ in a basis of KBSM(ST). Note that the above agree with the relation between KBSM($S^1\times S^2$) and KBSM(ST) that is presented in \cite{P}. Thus, we conclude that:

\[
{\rm KBSM(}S^1\times S^2 {\rm)}\, =\, \frac{{\rm KBSM(ST)}}{<a-bbm_{\pm}(a)>}\ \Leftrightarrow\ V_{\widehat{a}}\ =\  V_{\widehat{bbm_{\pm}(a)}}, \quad \ \forall\ a\ {\rm in\ a\ basis\ of\ KBSM(ST)}.
\]

These equations have very complicated formulations with the use of the standard basis $\mathcal{B}^{\prime}_{{\rm ST}}$. In Figure~\ref{tprbbm1} for example, we demonstrate the performance of a bbm on $t{t_1^{\prime}}^2\in \mathcal{B}^{\prime}_{{\rm ST}}$. To simplify the infinite system of equations, in \cite{D0} a new basis ${B}_{{\rm ST}}$ for the Kauffman bracket skein module of ST is presented, using the looping generator $t$ (for an illustration of elements in this basis see Figure~\ref{Newbasbr1}). More precisely, we have the following:

\begin{thm}\label{newbasis}
The following set is a basis for KBSM(ST):
\begin{equation}\label{basis}
B_{\rm ST}\ =\ \{t^{n},\ n \in \mathbb{N} \}.
\end{equation}
\end{thm}

\begin{figure}
\begin{center}
\includegraphics[width=4.5in]{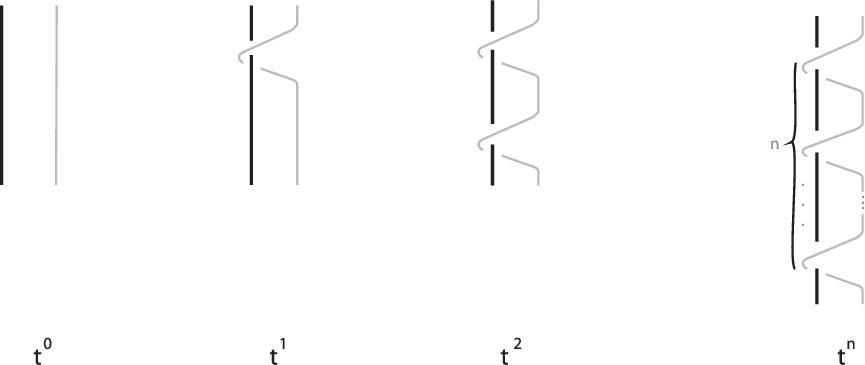}
\end{center}
\caption{Elements in the new basis of KBSM(ST).}
\label{Newbasbr1}
\end{figure}

\begin{remark}\rm
Note that elements in the basis $B_{\rm ST}$ have no crossings on the braid level, and thus, they are more ``natural'' and appropriate for our purpose.
\end{remark}

From Theorem~\ref{markov}, we observe that the equations $V_{\widehat{a}}\ =\  V_{\widehat{bbm_{\pm}(a)}}$ have particularly simple formulations with the use of the new basis $B_{\rm ST}$. Indeed, we have that $t^n\, \overset{bbm_{\pm}}{\rightarrow}\, t_1^{n}\, \sigma_1^{\pm 1}$. In Figure~\ref{tbbm1} we illustrate a bbm applied on $tt_1^2\in B_{\rm ST}$.

\begin{figure}[ht]
  \begin{subfigure}[b]{0.35\textwidth}
    \includegraphics[width=\textwidth]{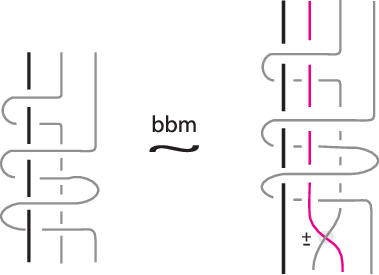}
    \caption{bbm performed on $t{t_1^{\prime}}^2 \in \mathcal{B}^{\prime}_{{\rm ST}}$.}
    \label{tprbbm1}
  \end{subfigure}
  \ \ \ \ \ \ \ \ \ \ \ \ \ 
  \begin{subfigure}[b]{0.35\textwidth}
    \includegraphics[width=\textwidth]{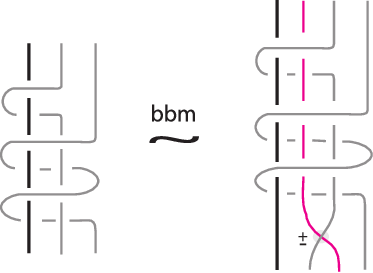}
    \caption{bbm performed on $tt_1^2 \in B_{ST}$.}
    \label{tbbm1}
  \end{subfigure}
\end{figure}

\section{The Kauffman bracket skein module of $S^1\times S^2$ via $TL_{1,n}$}\label{infs}

In this section we solve the infinite system of Equations~(\ref{eqbbm}), which is equivalent to computing the Kauffman bracket skein module of the lens spaces $S^1 \times S^2$. Recall that this infinite system of equations is obtained by performing braid band moves on elements in the basis $B_{\rm ST}$ of KBSM(ST) and by imposing to the generic invariant $V$ for knots and links in ST relations of the form $V_{\widehat{t^n}}\, =\, V_{\widehat{bbm(t^n)}}$, for all $n\in \mathbb{N}$, where $bbm(t^n)\, =\, t_1^n\, \sigma_1^{\pm 1}$. Recall also that the unknowns in the system are the $s_i$'s, coming from the fourth rule of the trace function in Theorem~\ref{tr}, that is, $tr(t^n)\, =\, s_n$, for all $n\in \mathbb{Z}$.

\smallbreak

We first observe that the equations obtained by applying the two types of bbm's on an element in $B_{\rm ST}$ other than the unknot, are different, and thus, both bbm's are needed in order to compute KBSM($S^1\times S^2$). In other words, we have that:

\smallbreak

{\it For $n\in \mathbb{N}\backslash\{0\}$, the equations $V_{\widehat{t^n}}\, =\, V_{\widehat{t_1^n\sigma_1}}$ and $V_{\widehat{t^n}}\, =\, V_{\widehat{t_1^n\sigma_1^{-1}}}$ are not equivalent.}

\smallbreak

Consider $t^n\in B_{{\rm ST}}$, where $n\geq 1$. Then, the equations obtained by performing bbm's on $t^n$ are:
\[
\begin{array}{lclc}
V_{\widehat{t^n}}\, =\, V_{\widehat{t_1^n\sigma_1^{\pm 1}}} & \overset{(+)-bbm}{\underset{(-)-bbm}{\Leftrightarrow}} & \begin{cases} s_n  & = \, \left(-\, \frac{1+u^2}{u}\right)\, u^{4n+2}\, tr(t_1^n\sigma_1)\\ s_n & = \, \left(-\, \frac{1+u^2}{u}\right)\, u^{4n-2}\, tr(t_1^n\sigma_1^{-1})\end{cases}
\end{array}
\]
We now deal with the equation obtained from a negative bbm and we underline expressions which are crucial for the next step.
\[
\begin{array}{lclc}
s_n & = & \left(-\, \frac{1+u^2}{u}\right)\, u^{4n-2}\, tr(t_1^n\underline{\sigma_1^{-1}}) & \overset{\sigma_1^{-1}=\sigma_1-(u-u^{-1})}{\Leftrightarrow}\\
&&&\\
s_n & = & \left(-\, \frac{1+u^2}{u}\right)\, u^{4n-2}\, \left[tr(t_1^n\sigma_1)\, -\, (u-u^{-1})\, tr(t_1^n)\right] & \Leftrightarrow\\
&&&\\
s_n & = & \underline{\left(-\, \frac{1+u^2}{u}\right)\, u^{4n+2}\, tr(t_1^n\sigma_1)}\,u^{-4} +\, \left(\frac{1+u^2}{u}\right)\, u^{4n-2}\,\, (u-u^{-1})\, tr(t_1^n) & \overset{(+)-bbm}{\Leftrightarrow}\\
&&&\\
s_n & = & u^{-4}\, s_n +\, \left(\frac{1+u^2}{u}\right)\, u^{4n-2}\,\, (u-u^{-1})\, tr(t_1^n) & \\
\end{array}
\]

Hence, for $n\in \mathbb{N}\backslash \{0\}$ we have that the equations obtained from the performance of the two types of bbm's are equivalent if and only if:
\begin{center}
\fbox{\begin{minipage}{15em}
\begin{equation}\label{eqtr}
\frac{u^4-1}{u^4} \cdot \left[s_n-u^{4n}\, tr(t_1^n) \right]\, =\, 0
\end{equation}
\end{minipage}}
\end{center}

We now evaluate $tr(t_1^n)$ for $n=1$ and we have:
\[
\begin{array}{lclcl}
tr(t_1) & = & tr(\underline{\sigma_1} t \sigma_1) \ =\ tr(t\underline{\sigma_1^2}) & = & (u-u^{-1})\, tr(t\sigma_1)+tr(t)\ =\\
&&&&\\
& = & (u-u^{-1})\, \underline{z}\cdot s_1 \ +\ s_1 & \overset{z=\frac{-1}{u(1+u^2)}}{=}& \frac{u^4+1}{u^2(1+u^2)}\cdot s_1,\\
\end{array}
\]

\noindent while substituting $n=1$ in Eq.~(\ref{eqtr}) we obtain that $tr(t_1)\, =\, \frac{1}{u^4}\cdot s_1$. Therefore, we have a contradiction and we conclude that in order to compute KBSM($S^1\times S^2$) we need to consider the effect of both types of bbm's on elements in $B_{{\rm ST}}$.

\subsection{Useful Lemmata}\label{infs1}

In this subsection we present a series of results toward the solution of the infinite system~(\ref{eqbbm}), by studying the effect of the braid band moves on the elements in the basis $B_{\rm ST}$. We recall first an ordering relation defined in \cite{DL2}. For that we shall need the notion of the {\it index} of a word $w$, denoted $ind(w)$.

\begin{defn}{\cite[Definition~1]{DL2}} \rm 
Let $w$ be a word in $\Sigma$ or in $\Sigma^{\prime}$. Then, the index of $w$, $ind(w)$, is defined to be the highest index of the $t_i$'s ($t_i^{\prime}$'s respectively) in $w$, by ignoring possible gaps in the indices of the looping generators and by ignoring the braiding parts. Moreover, the index of a monomial in $\sigma_i$'s is equal to $0$.
\end{defn}

\begin{defn}{\cite[Definition~2]{DL2}} \label{order}\rm
Let $w={t^{\prime}_{i_1}}^{k_1}\ldots {t^{\prime}_{i_{\mu}}}^{k_{\mu}}\cdot \beta_1$ and $u={t^{\prime}_{j_1}}^{\lambda_1}\ldots {t^{\prime}_{j_{\nu}}}^{\lambda_{\nu}}\cdot \beta_2$ in $\Sigma^{\prime}$, where $k_t , \lambda_s \in \mathbb{Z}$ for all $t,s$ and $\beta_1, \beta_2 \in H_n(q)$. Then, we define the following ordering in $\Sigma^{\prime}$:

\smallbreak

\begin{itemize}
\item[(a)] If $\sum_{i=0}^{\mu}k_i < \sum_{i=0}^{\nu}\lambda_i$, then $w<u$.

\vspace{.1in}

\item[(b)] If $\sum_{i=0}^{\mu}k_i = \sum_{i=0}^{\nu}\lambda_i$, then:

\vspace{.1in}

\noindent  (i) if $ind(w)<ind(u)$, then $w<u$,

\vspace{.1in}

\noindent  (ii) if $ind(w)=ind(u)$, then:

\vspace{.1in}

\noindent \ \ \ \ ($\alpha$) if $i_1=j_1, \ldots , i_{s-1}=j_{s-1}, i_{s}<j_{s}$, then $w>u$,

\vspace{.1in}

\noindent \ \ \  ($\beta$) if $i_t=j_t$ for all $t$ and $k_{\mu}=\lambda_{\mu}, k_{\mu-1}=\lambda_{\mu-1}, \ldots, k_{i+1}=\lambda_{i+1}, |k_i|<|\lambda_i|$, then $w<u$,

\vspace{.1in}

\noindent \ \ \  ($\gamma$) if $i_t=j_t$ for all $t$ and $k_{\mu}=\lambda_{\mu}, k_{\mu-1}=\lambda_{\mu-1}, \ldots, k_{i+1}=\lambda_{i+1}, |k_i|=|\lambda_i|$ and $k_i>\lambda_i$, then $w<u$,

\vspace{.1in}

\noindent \ \ \ \ ($\delta$) if $i_t=j_t\ \forall t$ and $k_i=\lambda_i$, $\forall i$, then $w=u$.

\end{itemize}

The ordering in the set $\Sigma$ is defined as in $\Sigma^{\prime}$, where $t_i^{\prime}$'s are replaced by $t_i$'s.
\end{defn}

Our goal is to express $bbm(t^n)\, =\, t_1^{n}\, \sigma_1^{-1}\, \widehat{\cong}\, t\, t_1^{n-1}\, \sigma_1$ for $n\in \mathbb{N}$, to sums of elements in $B_{{\rm ST}}$.

\begin{lemma}\label{l1}
For $n \in \mathbb{N}\backslash \{0\}$, the following holds in $TL_{1, n}$:
\[
\begin{array}{llcl}
{\rm i.} & t_1^n\, \sigma_1 & \widehat{\cong} & t^n\, \sigma_1\ +\ \underset{i=0}{\overset{n-1}{\sum}}\, (u-u^{-1})\, t^{i}\, {t_1}^{n-i}\\
&&&\\
{\rm ii.} & t_1^n\, \sigma_1^{-1} & \widehat{\cong} & t^n\, \sigma_1\ +\ \underset{i=1}{\overset{n-1}{\sum}}\, (u-u^{-1})\, t^{i}\, {t_1}^{n-i}\\
\end{array}
\]
\end{lemma}

\begin{proof}
We prove Lemma~\ref{l1}(i) by induction on $n\in \mathbb{N}$. The base of induction is $t_1\sigma_1\, \widehat{\cong}\, (u-u^{-1})\, t_1\ +\ t\, \sigma_1$, that holds. Assume now that the relations hold for $n$. Then, for $n+1$ we have that: 

\[
\begin{array}{lcl}
t_1^{n+1}\, \sigma_1 & = & t_1^{n}\, \underline{t_1\, \sigma_1}\ \widehat{\cong} \ (u-u^{-1})\, t_1^{n+1}\, +\, t\, \underline{t_1^{n}\, \sigma_1}\, \overset{ind.}{\underset{step}{=}}\\
&&\\
& = & (u-u^{-1})\, t_1^{n+1}\, +\, t^{n+1}\, \sigma_1\, +\, \underset{i=0}{\overset{n-1}{\sum}}\, (u-u^{-1})\, t^{i+1}\, {t_1}^{n-i}\, =\\
&&\\
& = & t^{n+1}\, \sigma_1\, +\, \underset{i=0}{\overset{n}{\sum}}\, (u-u^{-1})\, t^{i}\, {t_1}^{n+1-i}
\end{array}
\]
\end{proof}

We now express monomials of the form $t^{m}t_1^{n}\in \Sigma$, where $m, n\in \mathbb{N}$ in sums of elements of the form $t^{m}{t_1^{\prime}}^n\in \Sigma^{prime}$. In \cite{DL2}, a method for expressing monomials in $\Sigma$ to sums of monomials in $\Sigma^{\prime}$ (and vice-versa) is presented on the level of the generalized Hecke algebra of type B, $H_{1, n}$. This was achieved with the use of the ordering relation defined on the basic sets of $H_{1, n}$ (Definition~\ref{order}). More precisely, in \cite{DL2} it is shown that a monomial $\tau$ in $\Sigma$ can be written as a sum of elements in the set $\{{t^{\prime}_0}^{k_0}{t^{\prime}_1}^{k_1}\ldots {t^{\prime}_{m}}^{k_m}\ | \ k_i\ \geq\ k_{i+1},\ k_i \in \mathbb{Z}\setminus\{0\},\ \forall i \}$, such that the term $\tau^{\prime}$, which is obtained from $\tau$ by changing $t_i$ into $t_i^{\prime}$ for all $i$, is the highest order term in the sum. Recall now that the generalized Temperley-Lieb algebra of type B, $TL_{1, n}$, is a quotient of $H_{1, n}$ over the ideal generated by elements in Eq.~(\ref{ideal}) that only involves the $\sigma$'s, and since the only braiding generator in the monomials $bbm(t^n)$ is $\sigma_1^{\pm 1}$, it follows that the results presented in \cite{DL2} are also true in $TL_{1, n}$ for the monomials $bbm(t^n)=t_1^n\, \sigma_1^{\pm 1}$. In particular we have that for $n \in \mathbb{N}\backslash \{0\}$, the following holds in $TL_{1, n}$:
\[
\begin{array}{lclclcl}
t_1^n\, \sigma_1 & \widehat{\cong} & \underset{i=0}{\overset{n}{\sum}}\, a_i\, t^{i}\, {t_1^{\prime}}^{n-i} & {\rm \ \ and\ \ } &
t_1^n\, \sigma_1^{-1} & \widehat{\cong} & \underset{i=1}{\overset{n}{\sum}}\, a_i^{\prime}\, t^{i}\, {t_1^{\prime}}^{n-i}.
\end{array}
\]
\noindent where $a_i$ and $a_i^{\prime}$ coefficients for all $i$.

\smallbreak

We now deal with elements of the form $t^k\, {t^{\prime}_1}^{m}$, such that $k, m\in \mathbb{N}$ and $k\geq m>0$. 

\begin{lemma}\label{l2}
For $k, m\in \mathbb{N}$ and $k\geq m>0$, the following relation hold in $TL_{1, n}$:
\[
t^k\, {t^{\prime}_1}^{m}\ \widehat{\cong}\ -A^{-2}\, t^{k+m}\ -\ A^3\, t^{k+m-2}\ -\ A\, t^{k-1}\, {t_1^{\prime}}^{m-1}
\]
\end{lemma}

\begin{proof}
We consider $t^k\, {t^{\prime}_1}^{m}$ and we apply the Kauffman bracket skein relation as demonstrated in Figure~\ref{lem1}. Applying the skein relations and conjugation, we may express these elements in terms of elements in the basis $B_{\rm ST}$, together with elements of the form $t^{k-i}\, {t^{\prime}_1}^{m-i},\ i>0$, of lower order than $t^k\, {t^{\prime}_1}^{m}$. The process is repeated until we are left with elements in $B_{\rm ST}$.

\begin{figure}[H]
\begin{center}
\includegraphics[width=5.7in]{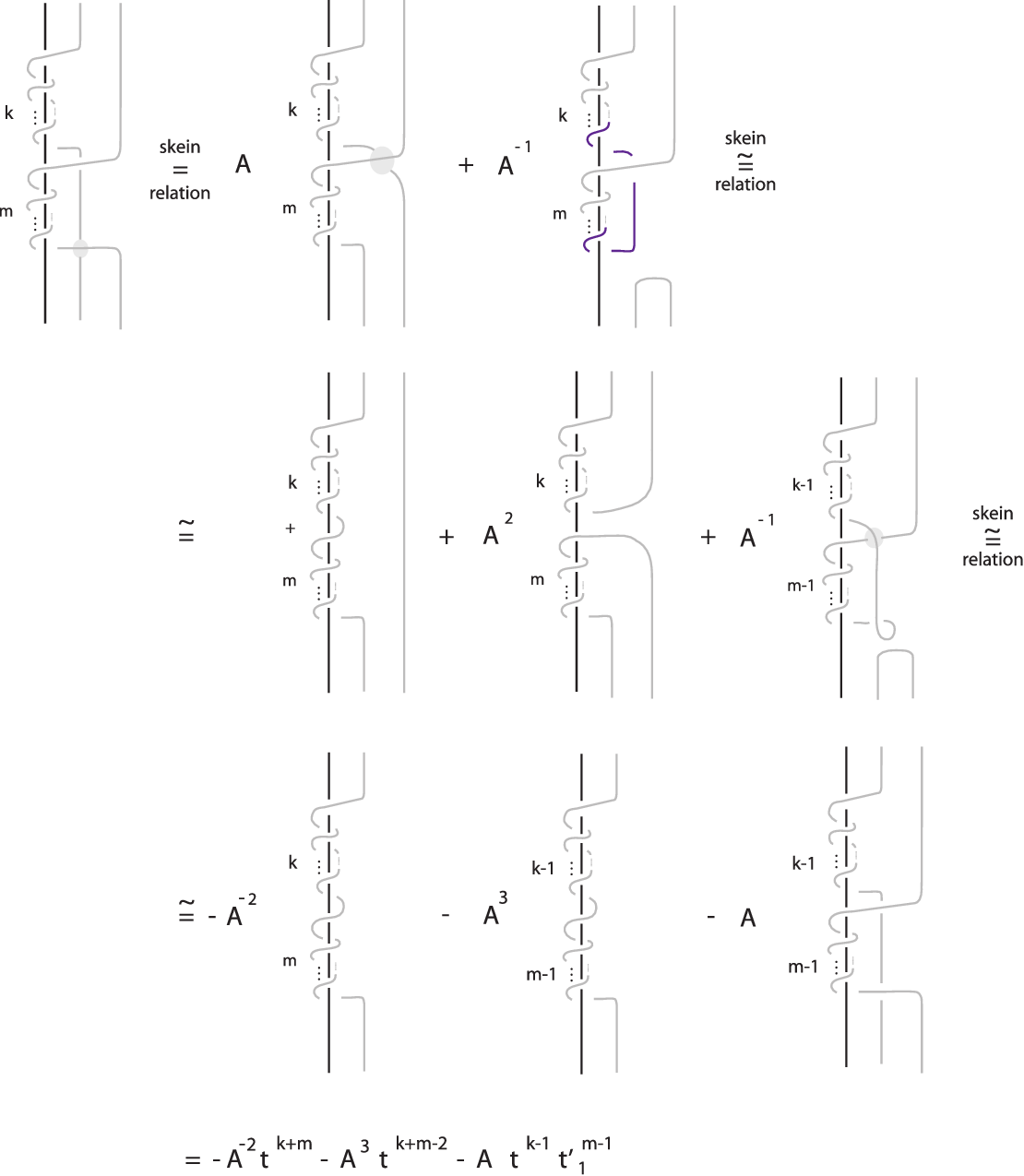}
\end{center}
\caption{The proof of Lemma~\ref{l2}.}
\label{lem1}
\end{figure}

\end{proof}

\begin{remark}\rm
In Lemma~\ref{l2}, we considered elements of the form $t^k\, {t^{\prime}_1}^{m}$, where $k\geq m>0$. If $k<m$, then we consider the closure operation on the mixed braids and we order the exponents, since the $t_i^{\prime}$'s are conjugates, i.e. $t^k\, {t^{\prime}_1}^{m}\, \widehat{=}\, t^m\, {t^{\prime}_1}^{k}$. Then, we braid the resulting loops again (see Figure~\ref{rem1}). 
\end{remark}

\begin{figure}[H]
\begin{center}
\includegraphics[width=5in]{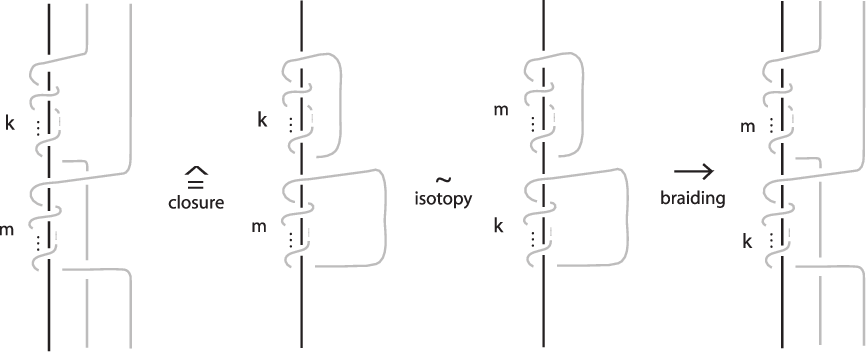}
\end{center}
\caption{Conjugation on the $t_i^{\prime}$'s.}
\label{rem1}
\end{figure}

\begin{cor}\label{cor1}
For $k, m\in \mathbb{N}$ and $k\geq m>0$, the following relation hold in $TL_{1, n}$:
\[
t^k\, {t^{\prime}_1}^{m}\ \widehat{\cong}\ -A^{-2}\, t^{k+m}\ +\ \underset{i=0}{\overset{\lfloor \frac{k+m-2}{2}\rfloor}{\sum}}\, a_i\, t^{2i},
\]
\noindent where $a_i$ are coefficients.
\end{cor}

\subsection{The infinite system}

By general position arguments, for computing the Kauffman bracket skein module of $S^1 \times S^2$, we may first consider the effect of negative bbm's on elements in $B_{\rm ST}$, obtain a spanning set for $B^{S^1\times S^2}_{-}\, :=\, \frac{B_{\rm ST}}{<\, a-bbm_{-}(a)\, >}$ and then study the effect of positive bbm's on elements in $B^{S^1\times S^2}_{-}$. With a slight abuse of notation, we denote the above as follows:

\[
{\rm KBSM}\left(S^1 \times S^2\right)\, =\, \frac{{\rm B}_{ST}}{<\, a-bbm_{\pm}(a)\, >}\, =\, \frac{{\rm B}_{ST}/<a-bbm_-(a)>}{<a-bbm_+(a)>}\, =\, \frac{B^{S^1\times S^2}_{-}}{<a-bbm_+(a)>}.
\]

Hence, we consider elements in $B_{ST}$ and perform negative braid band moves, obtaining an infinite system of equations, the solution of which corresponds to a basis for $B^{-}_{S^1\times S^2}$. We have the following:

\begin{thm}\label{thm1}
$B^{-}_{S^1\times S^2}$ is generated by the unknot $t^0$ and $t$. 
\end{thm}

\begin{proof}
We consider the infinite system of equations $V_{\widehat{t^n}}\, =\, V_{\widehat{t_1^n\, \sigma_1^{-1}}}$, for $n\in \mathbb{N}$. We have that:
\[
\begin{array}{lrcrcl}
{\rm For}\ n=0: & 1 \ \overset{bbm_-}{\rightarrow}\ \sigma_1^{-1} & \Rightarrow & V_{\hat{1}}\, =\, V_{\hat{\sigma_1^{-1}}} & \Rightarrow & 1 \ = \ \left(-\, \frac{1+u^2}{u} \right)\, u^{-2}\, tr(\sigma_1^{-1})\, \Rightarrow\\
&&&&&\\
& & & 1 & = & -\, \frac{1+u^2}{u^3}\, \left(z-(u-u^{-1}) \right)\ \overset{z=-\, \frac{1}{u(1+u^2)}}{\Rightarrow}\\
&&&&&\\
& & & 1 & = & 1 \\
&&&&&\\
{\rm For}\ n=1: & t \ \overset{bbm_-}{\rightarrow}\ t_1\, \sigma_1^{-1} & \Rightarrow & V_{\hat{t}}\, =\, V_{\widehat{t_1\, \sigma_1^{-1}}} & \Rightarrow & s_1 \ = \ \left(-\, \frac{1+u^2}{u} \right)\, u^{2}\, tr(t_1\, \sigma_1^{-1})\, \Rightarrow\\
&&&&&\\
& & & s_1 & = & -\, -\, \frac{1+u^2}{u}\, u^2\, z\, s_1\ \Rightarrow\\
&&&&&\\
& & & s_1 & = & s_1\\
\end{array}
\]

Hence, $t^0$, which corresponds to the unknot, and $t$ are {\it free} in $B^{-}_{S^1\times S^2}$.

\smallbreak

Consider now $n\in \mathbb{N}$ such that $n\geq 2$. Then:

\[
t^n \ \overset{bbm_-}{\longrightarrow} \ t_1^n\, \sigma_1^{-1} \ \widehat{\cong} \ t\, t_1^{n-1}\, \sigma_1 \ \widehat{\underset{{\rm [13]}}{\cong}} \ \underset{i=1}{\overset{n}{\sum}}\, a_i\, t^{i}\, {t_1^{\prime}}^{n-i} \ \underset{{\rm Cor.}~\ref{cor1}}{\widehat{{\cong}}} \ a_n\, t^{n}\ +\ \underset{i=0}{\overset{\lfloor \frac{n-2}{2}\rfloor}{\sum}}\, a_i\, t^{2i}
\]

\noindent where $a_i$ coefficients for all $i$.

\noindent Hence, $tr(t_1^n\, \sigma_1^{-1})\, =\, a_n\, s_{n}\ +\ \underset{i=1}{\overset{\lfloor \frac{n}{2} \rfloor}{\sum}}\, b_i\, s_{n-2i}$, for all $n\in \mathbb{N}$. Thus, by ignoring the coefficients, we have the following:

\[
\begin{array}{lcc}
V_{\hat{t^n}}\, =\, V_{\widehat{t_1^n\sigma_1^{-1}}} & \Rightarrow & s_n\, =\, \begin{cases} \underset{i=0}{\overset{(n-2)/2}{\sum}}\, s_{2i} &,\ {\rm for}\ n\ {even}\\
\underset{i=0}{\overset{(n-3)/2}{\sum}}\, s_{2i+1} &,\ {\rm for}\ n\ {odd}
\end{cases}\, \Rightarrow\\
\end{array}
\]

\[
\begin{array}{cc}
{\rm \underline{n-even}} & {\rm {\underline{n-odd}}}\\
&\\
s_2\, =\, s_0 & s_3\, =\, s_1\\
&\\
s_4\, =\, s_0\, +\, s_2 & s_5\, =\, s_1\, +\, s_3\\
&\\
\vdots & \vdots\\
&\\
s_{2k}\, =\, \underset{i=0}{\overset{k-1}{\sum}}\, s_{2i} & s_{2k+1}\, =\,  \underset{i=0}{\overset{k-1}{\sum}}\, s_{2i+1}\\
&\\
\Downarrow & \Downarrow\\
 s_{2k}\, \sim\, a\cdot s_0\ & \ s_{2k+1}\, \sim \, b\cdot s_1,
\end{array}
\]
\noindent where $a, b$ coefficients.

\smallbreak

\noindent Recall now that $tr(t)=s_1$ and $tr(t^0)=s_0$. The result follows.
\end{proof}

\begin{cor}
The set $\{t^0,\, t\}$ forms a basis for $B^{-}_{S^1\times S^2}$.
\end{cor}

We now consider the effect of positive braid band moves on the elements $t^0$ and $t$. We have the following result:

\begin{thm}\label{mthm}
The free part of KBSM$\left(S^1 \times S^2\right)$ is generated by the unknot and the torsion part is generated by $t$.
\end{thm}

\begin{proof}
We consider the elements in the basis of $B^{-}_{S^1\times S^2}$ and we study the effect of a positive braid band move on them. We have that:

\[
\begin{array}{lrcrcl}
{\rm For}\ n=0: & 1 \ \overset{bbm_+}{\rightarrow}\ \sigma_1 & \Rightarrow & V_{\hat{1}}\, =\, V_{\hat{\sigma_1}} & \Rightarrow & 1 \ = \ \left(-\, \frac{1+u^2}{u} \right)\, u^{-2}\, tr(\sigma_1)\, \Rightarrow\\
&&&&&\\
& & & 1 & = & -\, (1+u^2)\, u\, z \ \overset{z=-\, \frac{1}{u(1+u^2)}}{\Rightarrow}\\
&&&&&\\
& & & 1 & = & 1 \qquad {\rm (free\ part)}\\
&&&&&\\
{\rm For}\ n=1: & t \ \overset{bbm_+}{\rightarrow}\ t_1\, \sigma_1 & \Rightarrow & V_{\hat{t}}\, =\, V_{\widehat{t_1\, \sigma_1}} & \Rightarrow & s_1 \ = \ \left(-\, \frac{1+u^2}{u} \right)\, u^{6}\, tr(t_1\, \sigma_1)
\end{array}
\]

\noindent We evaluate $tr(t_1\, \sigma_1)$:
\[
\begin{array}{lclcl}
tr(t_1\, \sigma_1) & = & (u-u^{-1})\, tr(t_1)\, +\, tr(\sigma_1\, t) & = & (u-u^{-1})^2\, tr(t\, \sigma_1)\, +\, (u-u^{-1})\, tr(t)\, +\, z\, s_1\ =\\
&&&&\\
&&& = & (u^2-1+u^{-2})\, z\, s_1\ +\ (u-u^{-1})\, s_1
\end{array}
\]

\noindent Substituting $tr(t_1\, \sigma_1)$ in the equation above we obtain the following:

\begin{equation}\label{tor}
V_{\hat{t}}\, =\, V_{\widehat{t_1\, \sigma_1}}\ \Leftrightarrow\ (1-u^6)\, (1-u^2)\, s_1 \ =\  0.
\end{equation}

From Equation~\ref{tor} and Theorem~\ref{thm1} we have that $s_1=tr(t)$ generates the torsion part of KBSM($S^1 \times S^2$). The free part of KBSM($S^1 \times S^2$) is generated by the unknot (or the empty knot).
\end{proof}

Although we don't obtain a closed formula for the torsion part of KBSM($S^1 \times S^2$), Theorem~\ref{mthm} implies the following:

\begin{cor}
If $\pi$ denotes the natural projection 
$$\pi: {\rm KBSM}\left(S^1 \times S^2\right)\, \rightarrow\, {\rm KBSM}\left(S^1 \times S^2\right)/Tor,$$ 
\noindent then:
\[
\pi(t^{2n+1})\, =\, 0,\, \forall\, n\in \mathbb{N} \quad {\rm and} \quad {\rm KBSM}\left(S^1 \times S^2\right)/Tor\, =\, \mathbb{Z}[A^{\pm 1}].
\]
\end{cor}

\section{A diagrammatic approach via braids}\label{kbsmunbr}

In this section we present the diagrammatic approach for computing KBSM($S^1\times S^2$) via braids. Note that during the process that we will describe below, the braids ``lose'' their natural top-to-bottom orientation, giving rise to what we call {\it unoriented braids}. Unoriented braids were defined in \cite{D0} as standard braids by ignoring the natural top-to-bottom orientation (for an illustration see Figure~\ref{unbr1}). These tools seem promising in computing Kauffman bracket skein modules of arbitrary $3$-manifolds (see for example \cite{D2, D4}).

\begin{figure}[H]
\begin{center}
\includegraphics[width=3.7in]{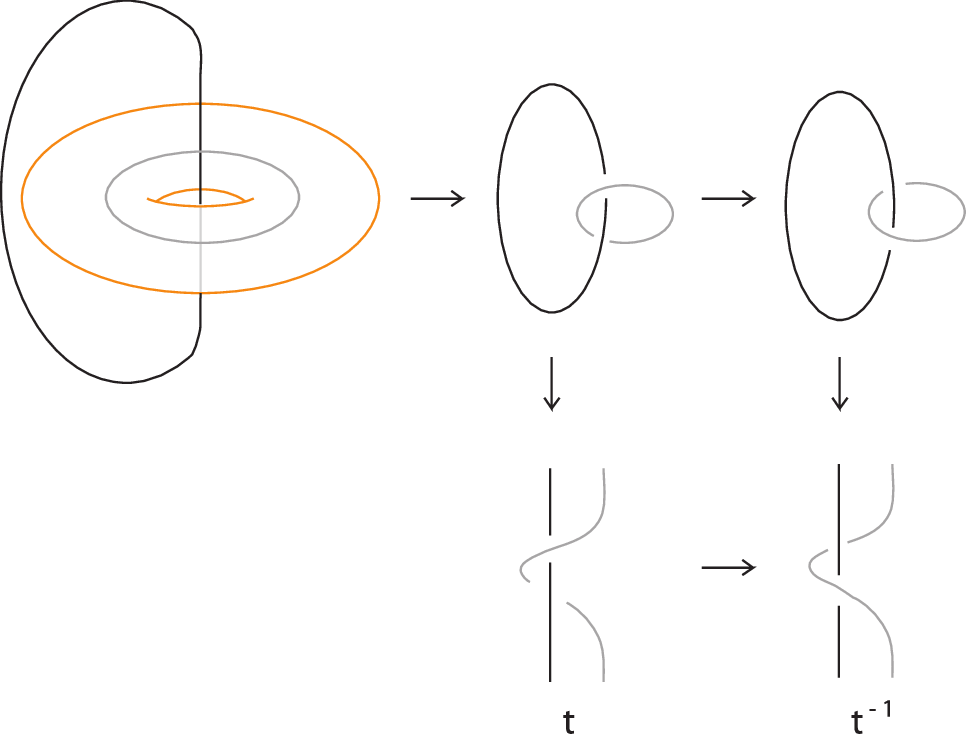}
\end{center}
\caption{Unoriented braids.}
\label{unbr1}
\end{figure}

In \cite{DL1} it is shown that in order to describe isotopy for knots and links in a c.c.o. $3$-manifold, it suffices to consider only one type of band moves (recall the discussion in \S~\ref{basics2}). In the algebraic approach we only considered the type $\beta$-band moves, that correspond to the braid band moves, and in the diagrammatic approach we will be using the $\alpha$-type band moves (recall Figure~\ref{bmov}). Our starting point again is the following:
\begin{equation}\label{infsys2}
{\rm KBSM}(S^1\times S^2)\, =\, {\rm KBSM(ST)}\ /\, <\, \alpha-{\rm type\ band\ moves}\, >\ =\ B_{\rm ST}\ /\, <\, \alpha-{\rm type\ band\ moves}\, >
\end{equation}
\noindent, that is, we study the effect of the $\alpha$-type band moves on elements in the $B_{\rm ST}$ basis of KBSM(ST). We denote the effect of $\alpha$-type band moves on $t^n\in B_{\rm ST}$ as $bm(t^n)=A^{6}x_n$, where $x_n$ is illustrated in Figure~\ref{xn1}. Obviously, $x_1\, =\, \hat{t}$.

\begin{figure}[H]
\begin{center}
\includegraphics[width=4.2in]{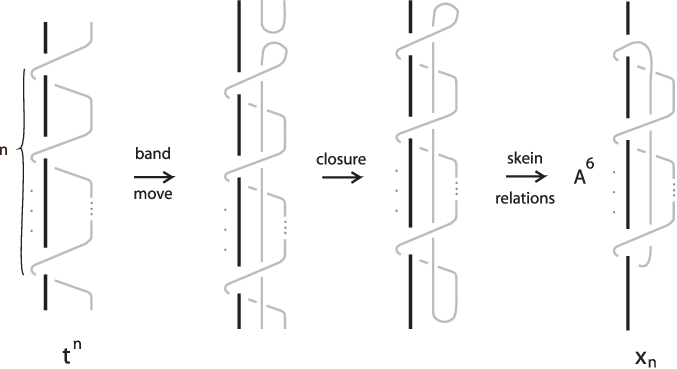}
\end{center}
\caption{Performing band moves on elements in $B_{{\rm ST}}$.}
\label{xn1}
\end{figure}

We now return to the infinite system of equations (\ref{infsys2}) and we have the following:

\begin{equation}\label{trsn}
t\ \overset{{\rm band}}{\underset{{\rm move}}{\longrightarrow}}\ A^6\, x_1\ \Leftrightarrow\ (1-A^6)\, t\ =\ 0.
\end{equation}

\noindent Hence, $t$ produces torsion in KBSM($S^1\times S^2$).

\begin{nt}\rm
We denote by $\hat{t}$ the closure of the looping generator $t$ and by $\widehat{t^m}$ the closure of $m$-looping generators in ST. Moreover,  $x_n\, \widehat{t^m}$, $n, m\in \mathbb{N}$, denotes $x_n$ followed by $m$ copies of $\hat{t}$. For an illustration see Figure~\ref{znt1}.
\end{nt}

\begin{figure}[H]
\begin{center}
\includegraphics[width=1.2in]{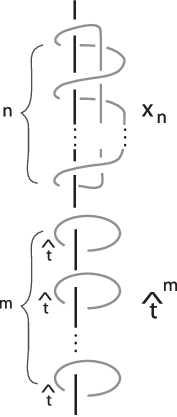}
\end{center}
\caption{The elements $x_n\, \widehat{t^m}$ for $n, m\in \mathbb{N}$.}
\label{znt1}
\end{figure}

In order to clarify the steps needed toward the computation of KBSM(ST) using the diagrammatic method via braids, we also demonstrate how we obtain the equation for $t^2$. After the performance of the band move on $t^2$, we apply the Kauffman bracket skein relation on the resulting $A^{6}x_2$ and we write $x_2$ in terms of $\widehat{t}$'s. Applying the skein relation then again to the $\widehat{t}$'s, we express them in terms of the $t_i$'s, obtaining the equation of the infinite system. For an illustration see Figure~\ref{x22}. Indeed, we have that:

\[
\begin{array}{rcccl}
t^2 & \rightarrow & bm(t^2)\, =\, A^{6}\, x_2 & \widehat{\cong} & -A^{10}\, \widehat{t}^2\, -A^{8}\\
&&&&\\
{\widehat{t}}^2 & \overset{\rm braiding}{\rightarrow} & tt_1^{\prime} & \widehat{\cong} & -A^{-2}\, t^2\, -\, A^{2}
\end{array}
\]

Hence, we obtain the following equation:

\[
(1-A^8)\, t^2\ =\ -A^8\, (1-A^4).
\]

\begin{figure}[H]
\begin{center}
\includegraphics[width=6.2in]{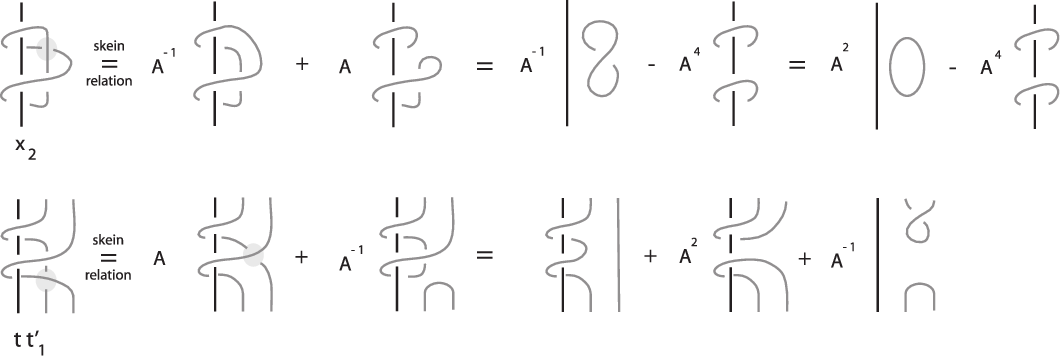}
\end{center}
\caption{Expressing $x_2$ in terms of elements in $B_{{\rm ST}}$.}
\label{x22}
\end{figure}

We observe that in order to obtain the equations of the infinite system, a recursive formula for $x_n$ is needed. We have the following:

\begin{lemma}\label{l3}
For $n\in\mathbb{N}$ such that $n\geq 2$, the following relations hold in KBSM($S^1\times S^2$):
\[
x_n\ \sim\ -A^8\, x_{n-2}\ -\ A^{4}\, x_{n-1}\, \widehat{t}.
\]
\end{lemma}

\begin{figure}[H]
\begin{center}
\includegraphics[width=5.9in]{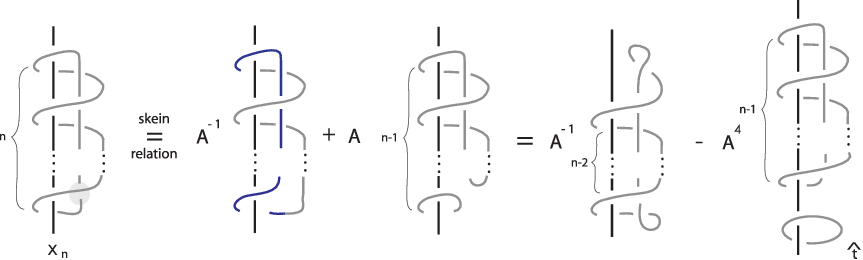}
\end{center}
\caption{The proof of Lemma~\ref{l3}.}
\label{xnfor1}
\end{figure}

Note now that $\widehat{t}$ won't affect the process illustrated in Figure~\ref{xnfor1} for $x_{n-1}\, \widehat{t}$. Indeed, we have that
\[
x_{n-1}\, \widehat{t}\, \sim\, -A^{8}\, x_{n-3}\widehat{t}\ -\ A^{4}\, x_{n-2}\, \left(\, \widehat{t}\, \right)^2.
\] 

Moreover, for $n=2$ we have that $x_2\, \sim\, -A^4\, \widehat{t}^2\, -A^2$.

\smallbreak

We now present an ordering relation for elements in the set $D\, :=\, \{x_n\, \widehat{t}\, ^m \}_{n, m\in \mathbb{N}}$.

\begin{defn}\rm \label{ordef}
Let $\alpha\, =\, x_n\, \widehat{t}\, ^m$ and $\beta\, =\,  x_k\, \widehat{t}\, ^l$. Then:
\smallbreak
\begin{itemize}
\item[i.] If $n+m<k+l$, then $\alpha<\beta$.
\smallbreak
\item[ii.] If $n+m=k+l$, then:
\begin{itemize}
\item[a.] If $n<k$, then $\alpha<\beta$.
\smallbreak
\item[b.] If $n=k$, then $\alpha=\beta$.
\end{itemize}
\end{itemize}
\end{defn}

\begin{prop}
The set $D\, :=\, \{x_n\, \widehat{t}\, ^m \}_{n, m\in \mathbb{N}}$ equipped with the ordering relation given in Definition~\ref{ordef} is a totally ordered and a well-ordered set.
\end{prop}

\begin{proof}
In order to show that the set $D$ is a totally ordered set when equipped with the ordering given in Definition~\ref{ordef}, we need to show that the ordering relation is antisymmetric, transitive and total. We only show that the ordering relation is transitive. Antisymmetric property follows similarly and totality follows from Definition~\ref{ordef}, since all possible cases have been considered.

\smallbreak

Let $\alpha\, =\, x_n\, \widehat{t}\, ^m,\ \beta\, =\,  x_k\, \widehat{t}\, ^l,\ \gamma\, =\,  x_p\, \widehat{t}\, ^q$, such that $\alpha<\beta$ and $\beta<\gamma$. We prove that $\alpha<\gamma$. Since $\beta<\gamma$, we have that:
\smallbreak
\begin{itemize}
\item[(a.)] either $n+m<k+l$, and since $\beta<\gamma$, we have that $k+l\leq p+q\, \Rightarrow\, n+m<p+q\, \Rightarrow\, \alpha<\gamma$,
\smallbreak
\item[(b.)] or $n+m=k+l$ such that $n<k$. Then, since $\beta<\gamma$, we have that:
\begin{itemize}
\item[(i.)] either $k+l<p+q$, same as in case (a.),
\smallbreak
\item[(ii.)] or $k+l=p+q$ such that $k<p$. Then, $n<k<p\, \Rightarrow\, \alpha<\gamma$.
\end{itemize}
\end{itemize}

We conclude that the ordering relation is transitive. The minimum element in $D$ is $x_0\widehat{t}\, ^0$, i.e. the unknot, and hence $D$ is a well-ordered set.
\end{proof}

The next step is to convert the $x_n$'s to elements in the $B_{\rm ST}$ basis of KBSM(ST) for all $n\in \mathbb{N}$. For that, we need the following lemmas:

\begin{lemma}\label{lemxn}
For $n\in \mathbb{N}$ such that $n\geq 2$, the following relations hold in KBSM($S^1 \times S^2$):
\begin{equation}\label{eqx}
x_n\, =\, (-1)^{n-1}\, A^{4n-4}\, \widehat{t}\, ^n\, +\, \begin{cases} \underset{i=0}{\overset{\frac{n-2}{2}}{\sum}}\, a_i\, \widehat{t}\, ^{2i} &,\ {\rm for}\ n\ {\rm even}\\ &\\ \underset{i=0}{\overset{\lfloor \frac{n-2}{2} \rfloor}{\sum}}\, b_i\, \widehat{t}\, ^{2i+1} &,\ {\rm for}\ n\ {\rm odd} \end{cases} 
\end{equation}
\end{lemma} 

\begin{proof}
We prove Lemma~\ref{lemxn} by strong induction on $n$. The case $n=2$ is the base of induction and we have that $x_2\, \widehat{\cong}\, -A^{4}\, \widehat{t}\, ^2\, -A^{2}$. Hence, relation~(\ref{eqx}) holds for $n=2$. Assume that relations~(\ref{eqx}) hold for all $n\in \mathbb{N},\, n>2$ up to $k\in \mathbb{N}$. Then, for $k+1$ we have:
\[
\begin{array}{lclc}
x_{k+1} & \overset{L.~\ref{l3}}{=} & -A^{8}\, \underline{x_{k-1}}\, -\, A^{4}\, \underline{x_{k}}\, \widehat{t} & \overset{ind.}{\underset{step}{=}}\\ 
&&&\\
& = & -A^{8}\, \left((-1)^{k-1}\, A^{4(k-1)-4}\, \widehat{t}\, ^{k-1}\, +\, \sum\, {\rm l.o.t.} \right) & +\\
&&&\\
& + & -A^{4}\, \left((-1)^{k-1}\, A^{-4k+4}\, \widehat{t}\, ^{k+1}\, +\, \sum\, {\rm l.o.t.} \right) & =\\
&&&\\
& = & (-1)^{k}\, A^{-4k}\, \widehat{t}\, ^{k-1}\, +\, (-1)^{k+2}\, A^{-4k}\, \widehat{t}\, ^{k+1}\, +\, \sum\, {\rm l.o.t.} & =\\
&&&\\
& = & (-1)^{k+2}\, A^{-4(k+1)+4}\, \widehat{t}\, ^{k+1}\, +\, \sum\, {\rm l.o.t.}\\ 
\end{array}
\]
\end{proof}

Note that $\widehat{t^{k}}$ corresponds to the monomial $tt_1^{\prime} \ldots t_{k-1}^{\prime}\, =\, \tau^{\prime}_{0, k-1}$, namely, a monomial in the $\mathcal{B}^{\prime}_{{\rm ST}}$ basis of KBSM(ST). We now convert these monomials to elements in the $B_{{\rm ST}}$ basis. We also note that this is done in \cite[\S~3]{D0} via the ordering relation of Definition~\ref{order}.

\begin{lemma}\label{tis}
The following relations hold in KBSM(ST), and hence in KBSM($S^1 \times S^2$), for $k, m\in \mathbb{N}$ such that $k, m\geq 2$:
\[
\beta\, {t_i^{\prime}}^{k}\, {t_{i+1}^{\prime}}^{m}\, \widehat{\cong}\, -A^{-2}\, \beta\, {t_{i}^{\prime}}^{k+m}\, -\, A^{3}\, \beta\, {t_i^{\prime}}^{k+m-2}\, -\, A\, \beta\, {t_i^{\prime}}^{k-1}\, {t_{i+1}^{\prime}}^{m-1}
\]
\noindent where $\beta$ is a monomial of $t_i^{\prime}$'s with index at most $i-1$.
\end{lemma}

\begin{proof}
The proof of Lemma~\ref{tis} is an immediate result of Lemma~\ref{l2} (recall also Figure~\ref{lem1}).
\end{proof}

Lemma~\ref{tis} implies the following:

\begin{prop}\label{prop2}
For $n\in \mathbb{N}$, such that $n>0$, the following relations hold:
\[
\widehat{t}\, ^n\ \widehat{\underset{{\rm skein}}{\cong}}\ \left( -A^{-2} \right)^{n}\, t^{n+1}\, +\, \underset{i=0}{\overset{n}{\sum}}\, a_i\, t^{i},
\]
\noindent where the $a_i$'s are coefficients.
\end{prop}

\begin{proof}
Consider the monomial $\widehat{t}\, ^n\, :=\, tt^{\prime}_1\ldots t^{\prime}_{n-1}$ and apply Lemma~\ref{tis} on each pair of loop generators with consecutive indices. Each time Lemma~\ref{tis} is applied, the coefficient $-A^{-2}$ will appear on the highest order term of the resulting sum. In total, there are $n-1$ pairs of $t_i^{\prime}$'s where Lemma~\ref{tis} is to applied. The result follows.
\end{proof}

We are now ready to present the main theorem of this section.

\begin{thm}\label{mth}
The following relations hold in KBSM($S^1 \times S^2$):
\[
\left(1-A^{2n+4} \right)\, t^n\, =\, \begin{cases} \underset{i=0}{\overset{(n-2)/2}{\sum}}\, a_i\, t^{2i} &,\ {\rm for}\ n\ {\rm even}\\ &\\ \underset{i=0}{\overset{(n-3)/2}{\sum}}\, b_i\, t^{2i+1} &,\ {\rm for}\ n\ {\rm odd} \end{cases}
\]
\end{thm}

\begin{proof}
\[
\begin{array}{cclclc}
t^n & \overset{bm}{\rightarrow} & A^{6}\, \underline{x_n} & \overset{L.~\ref{lemxn}}{=} & A^{6}\, \left[(-1)^{n-1}\, A^{4n-4}\, \underline{\widehat{t}\, ^n}\, +\, \sum\, {\rm l.o.t.}\right] & {\widehat{\underset{{\rm Prop.~\ref{prop2}}}{\cong}}}\\
&&&&&\\
&&& \widehat{\underset{{\rm skein}}{\cong}} & (-1)^{n-1}\, A^{4n+2}\, \left( -A^{-2} \right)^{n-1}\, \left(t^{n}\, +\, \underset{i=0}{\overset{n-2}{\sum}}\, a_i\, t^{i} \right)\, +\, \sum\, {\rm l.o.t.} & =\\
&&&&&\\
&&& = & (-1)^{2n-2}\, A^{4n+2-2n+2}\, t^{n}\, +\, +\, \sum\, {\rm l.o.t.} & =\\
&&&&&\\
&&& = & A^{2n+4}\, t^{n}\, +\, +\, \sum\, {\rm l.o.t.}  &\\
\end{array}
\]
\noindent where ``l.o.t.'' stands for lower ordered terms than $t^n$.
\end{proof}

Considering now the result of Theorem~\ref{mth} on the infinite system of equations~(\ref{infsys2}), a solution of which corresponds to computing KBSM($S^1 \times S^2$), we have that the elements of the form $t^{2n+1}$ for $n\in \mathbb{N}$, correspond to the torsion part of KBSM($S^1\times S^2$), since these elements can be written as $a\cdot t$ and $t$ produces torsion in $S^1\times S^2$ (recall relation~(\ref{trsn})). Moreover, we have that the free part of KBSM($S^1\times S^2$) is generated by the unknot. In other words, we have shown that:

\begin{thm}
\[
{\rm KBSM}\left(S^1 \times S^2\right)\ =\ \mathbb{Z}[A^{\pm 1}]\, \oplus\, \underset{i=0}{\overset{\infty}{\bigoplus}}\, \frac{\mathbb{Z}[A^{\pm 1]}}{(1-A^{2i+4})}
\]
\end{thm}

Our results agree with that on \cite{HP1}, where KBSM($S^1\times S^2$) is computed with the use of an appropriate basis of Chebyshev polynomials presented in \cite{L}.

\section{Conclusions}

In this paper we compute the Kauffman brakcet skein module of $S^1\times S^2$ using two different methods: the algebraic method based on the generalized Temperley-Lieb algebra of type B, and the diagrammatic method via braids. The algebraic approach was first implemented for the computation of HOMFLYPT skein modules of 3-manifolds via the generalized Hecke algebras of type B (see for example \cite{DL3, DL4, DLP, D3, DGLM}). It was then extended for the computation of Kauffman bracket skein modules of $3$-manifolds via the generalized Temperley-Lieb algebra of type B (see \cite{D0}). The diagrammatic method for computing Kauffman bracket skein modules has been successfully implemented for the case of the handlebody of genus two (\cite{D2}), for the complement of $(2, 2p+1)$- torus knots (\cite{D4}) and for the case of the lens spaces $L(p,q)$ (\cite{D0}). It is also worth mentioning that these techniques have been applied toward the computation of various skein modules of different families of `knotted objects' in (various) 3-manifolds. In particular, in \cite{D7} the algebraic background toward the computation of the HOMFLYPT skein module of the solid torus for singular links is presented (see also \cite{D9}) and in \cite{D8} the notion of skein modules is extended for {\it knotoids}. Finally, in \cite{D6} the algebraic set up toward the computation of skein modules of {\it tied links} in various $3$-manifolds is presented (see also \cite{D5}).

%\bigbreak

%\noindent {\bf Conflict of Interest:} The author states that there is no conflict of interest.

\end{document}